\documentclass[oneside,english]{amsart}
\usepackage{amssymb}
\usepackage{amsfonts}
\usepackage{amsmath}
\usepackage{graphicx}
\usepackage{hyperref}
\usepackage{float}
\usepackage{epstopdf}
\usepackage{color}
\usepackage{comment}
\usepackage{bm}
\usepackage{soul}

\makeatletter
\numberwithin{equation}{section}

\makeatother

\usepackage{babel}

\usepackage{fancyhdr}

\newtheorem{Def}{Definition}
\newtheorem{Thm}{Theorem}
\newtheorem{Lemma}{Lemma}
  
\newtheorem{Prop}{Proposition}

\newcommand{\C}{\mathbb{C}}
\newcommand{\p}{\mathbb{P}^1}
\newcommand{\R}{\mathbb{R}}
\newcommand{\J}{\mathcal{J}}
\newcommand{\E}{\mathcal{E}}

\makeatletter
\newcommand{\doubletilde}[1]{{%
  \mathpalette\double@Tilde{#1}%
}}
\newcommand{\double@Tilde}[2]{%
  \sbox\z@{$\m@th#1\Tilde{#2}$}%
  \ht\z@=0.95\ht\z@
  \Tilde{\box\z@}%
}

\makeatother

\begin{document}

\title[\tiny{Laplacians on Julia sets of rational maps}]{Laplacians on Julia sets of rational maps}
\author{Malte S. Ha\ss ler}
\address{Jacobs University Bremen, Bremen, 28759, Germany}
\curraddr{} \email{m.hassler@jacobs-university.de}
\thanks{Research supported by an undergraduate research internship program held at Cornell university in summer 2019}

\author{Hua Qiu}
\address{Department of Mathematics, Nanjing University, Nanjing, 210093, P. R. China.}
\curraddr{} \email{huaqiu@nju.edu.cn}

\author{ Robert S. Strichartz }
\address{Department of Mathematics, Cornell University, Ithaca, 14853, U.S.A.}
\curraddr{} \email{str@math.cornell.edu}



\date{}

\dedicatory{}
\begin{abstract}
The study of Julia sets gives a new and natural way to look at fractals. When mathematicians investigated the special class of Misiurewicz's rational maps, they found out that there is a Julia set which is homeomorphic to a well known fractal, the Sierpinski gasket. In this paper, we apply the method of Kigami to give rise to a new construction of Laplacians on the Sierpinski gasket like Julia sets with a dynamically invariant property.
\end{abstract}
\maketitle

\section{Introduction}

Recall that the familiar Sierpinski gasket (SG) is generated in the following way. Starting with a triangle one divides it into four copies, removes the central one, and repeats the iterated process. 
One aspect of the study of this fractal stems from the analytical construction of a Laplacian developed by Jun Kigami in 1989 \cite{kig89}. Since then, the analysis on SG has been extensively investigated from various viewpoints.
The theory has been extended to some other fractals \cite{kig93}, too. And for some of them one could say they are more \textit{invented} like SG  than \textit{discovered}. A standard reference to this topic is the book of Robert S. Strichartz \cite{str}.

When mathematicians studied the behaviour of polynomial maps under iteration, at first, they did not think about fractals. Given a starting point $z_0 \in \C$ and a polynomial $P(z)$ they wanted to know, whether the sequence of iterations $P(z_0),P(P(z_0))...$ converges. They named the set of points that show this behaviour the \textit{filled Julia set} and its boundary the \textit{Julia set} after the French mathematician Gaston Julia. For example, the Julia set of $P(z)=z^2$ is just the unit circle. But the slightly different polynomial $P(z)=z^2-1$ has a more complicated structure, see Figure \ref{basilica}.

\begin{figure}[H]
\begin{center}
\includegraphics[scale=0.15]{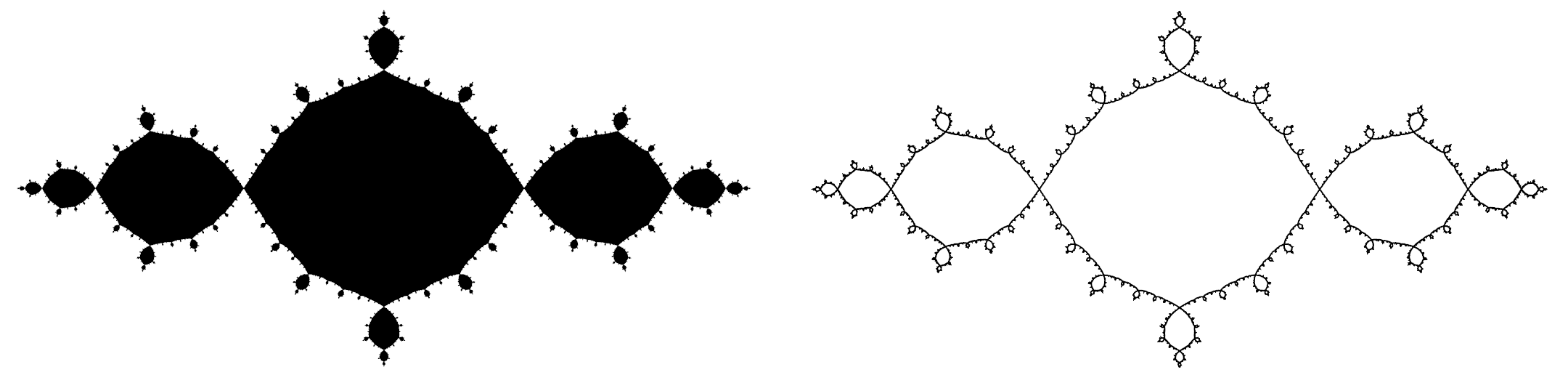}
\end{center}
\caption{The basilica filled Julia set and its boundary}\label{basilica}
\end{figure}

One can see a self-similar structure of this set. The fractal structure is a result from complex dynamics and the often chaotic behaviour of maps under iteration. 
The construction of a Laplacian for the basilica Julia set has been performed in \cite{bas} inspired by the theory of external rays. \cite{tar} and \cite{two} have built up on it and gave a construction for other certain quadratic polynomials. The theory can be generalized to higher degree in certain cases \cite{mywebsite}. An advantage is that the Laplacian can be made invariant to the polynomial map. Hence, one has an interesting connection of harmonic analysis on fractals with complex dynamics. For example, this is useful to understand the spectrum of the Laplacian. 

\begin{figure}[H]
\begin{center}
\includegraphics[scale=0.15]{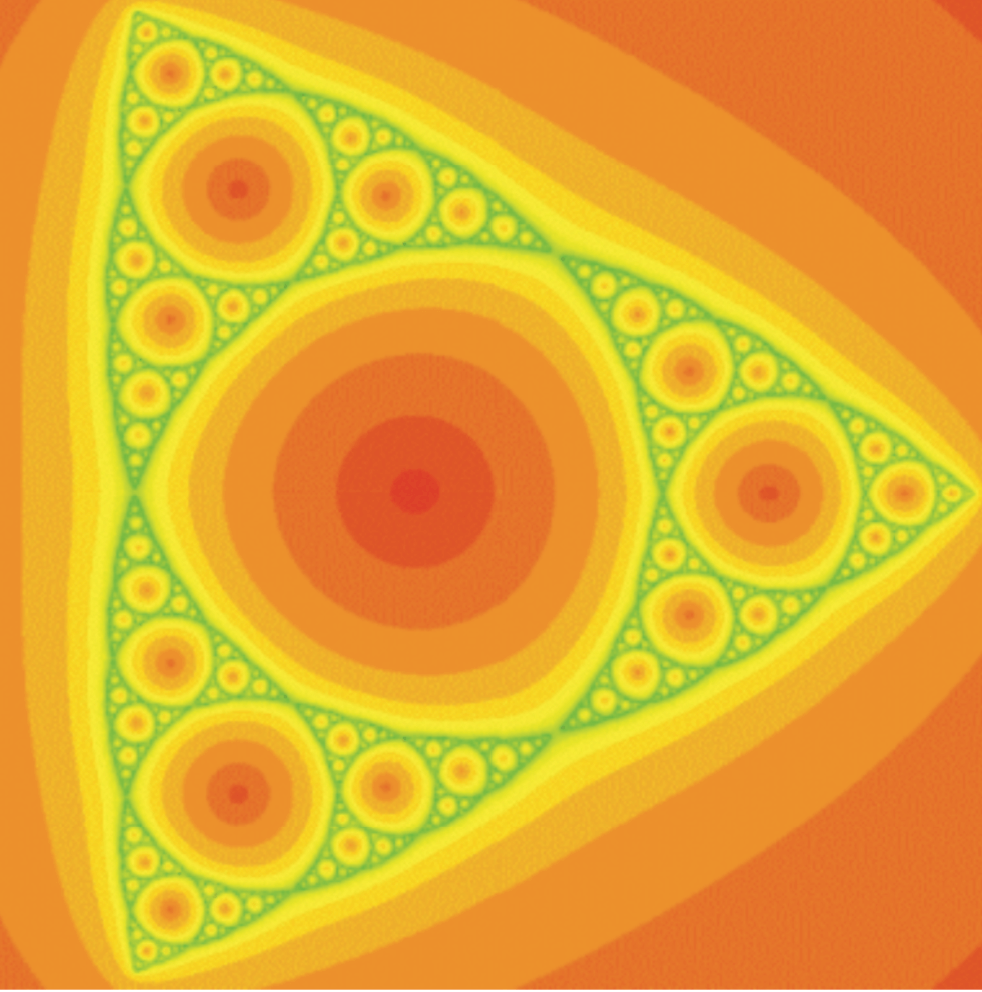}
\end{center}
\caption{The Julia set of $z^2-\frac{\lambda}{z}$ with $\lambda\approx 0.59267$}\label{degree3}
\end{figure}

Going further, one wants to study Julia sets of rational maps. Surprisingly, there exists a rational map of degree $3$ whose Julia set is homeomorphic to the Sierpinski gasket \cite{dev}! See Figure \ref{degree3}. This shows that SG is not necessarily a man-made fractal, it naturally occurs in the world of complex dynamics equipped with a complex valued map that will help to understand the fractal further. The goal of this paper is to define energy forms and Laplacians on this type of Julia sets that respect the dynamical properties.

In Section $2$ we will define Julia sets for rational maps and present the necessary theoretical background. In Section $3$ we present the work by Devaney et. al. \cite{dev} who made a topological description of certain Julia sets that look similar to the Sierpinski gasket, show how the rational map ``acts'' on SG and how to define the graph approximation. In Section $4$ we construct the standard energy for SG that is invariant under the rational map.
The next step of constructing a Laplacian is to determine a measure. In the following section we refer to the work of Denker and Urbanski that have studied the ergodic and invariant measures on Julia sets of Misiurewicz rational maps \cite{DU}. Applying this to our case will lead to the definition of a dynamically invariant Laplacian on SG. In Section $6$ we construct an iterated function system to describe the collection of all symmetric and invariant energy forms in the upcoming part.
The rational maps give rise to a large class of fractals similar to SG and we will generalize the construction of Laplacians in the final section. 

\section{Julia sets of rational maps}

Complex dynamics is not restricted to polynomial maps. One can also investigate the dynamics of rational functions $R(z)=\frac{P(z)}{Q(z)}$ where $P$ and $Q$ are complex valued polynomials. One might react critical to the case when $z$ is a root of $Q$, since $z$ is mapped to $\infty$ by $R$, but for the iteration of functions, $\infty$ is not a special point. Hence, one deals with maps $R: \p\mapsto\p$, where $\p :=\C\cup\{\infty\}$ is the Riemann sphere. Indeed, $\p$ can be identified as the usual 2-sphere which also provides a metric.

Now one wants to define a Julia set for the rational map. The definition for polynomials by bounded orbits does not work anymore. Instead, one defines Julia sets by \textit{normal families}. The definition and further mentioned properties are from \cite{bla}, which gives a rigorous introduction into the dynamics of rational maps. A standard reference is also given by Beardon \cite{bear}.

\begin{Def}[\cite{bla} p.89]
Let $U$ be an open subset of $\p$ and $\mathfrak{F}=\{f_i \, |i\in I\}$ a family of meromorphic functions on $\p$ defined on $U$ ($I$ is any index set). The family $\mathfrak{F}$ is a \textit{normal family} if every sequence $f_n$ contains a subsequence $f_{n_j}$ which converges uniformly on compact subsets of $U$.
\end{Def}

\begin{Def}
The \textit{Fatou set} $\mathcal{F}$ of a rational map $R: \p\mapsto\p$ is the set of
points that have a neighborhood on which the sequence of iterates $R^n$ forms a normal family. The \textit{Julia set} $\mathcal{J}$ is the set of points that have no such neighbourhood.
\end{Def}

We will not discuss the origin of this definition. It is important that it coincides with the definition for polynomials and similar properties of Julia sets still hold: the Julia set is compact and completely invariant, meaning that
\begin{equation}
    R(\J)=\J=R^{-1}(\J).
\end{equation}

The formal definition of a Julia set is not intrinsically useful to decide whether a point belongs to the Julia set or not. For periodic points this can be decided rather easily with the following definition and proposition. Together with the invariance property one can conclude for more points to be in the Julia set. 

\begin{Def}[\cite{bla} p.93]
The periodic orbit $O^+(z_0)$ of a periodic point consists of all points $R^k(z_0)$ for $1\leq k<n$ and $R^n(z_0)=z_0$. Let $\mu=(R^n)'(z_0)$. A periodic orbit is:
\begin{itemize}
    \item \textit{attracting} if $0<|\mu|<1$,
    \item \textit{superattracting} if $\mu=0$,
    \item \textit{repelling} if $|\mu|>1$,
    \item \textit{indifferent} if $|\mu|=1$.
\end{itemize}
\end{Def}

\begin{Prop}
If $O^+(z_o)$ is a (super)attracting periodic orbit, then it is contained in $\mathcal{F}$. If it is a repelling orbit, then it is contained in $\mathcal{J}$.
\end{Prop}

One should note that $\mu$ is a constant for the orbit independent of the choice of $z_0$, this can be seen by repeatedly applying the chain rule:
\begin{equation}
\label{chain}
(R^n)'(z_0)=R'(R(z_0))\cdot R'(R^2(z_0))\cdots R'(R^n(z_0)).
\end{equation}

A class of rational maps we will focus on are so-called \textit{Misiurewicz rational maps}. They are defined by the special properties of the \textit{critical points}, i.e. all points $z\in \C$ satisfying $R'(z)=0$, which always play an essential role to understand the dynamics. Denote the set of critical points by $CP(R)$. 
And call $\Omega(R)$ the $\omega$-limit set of $CP(R)$, that means $z \in \Omega(R)$ iff there exists a $c\in CP(R)$ and an unbounded sequence $n_k$ of positive integers such that $z=\lim_{k\to\infty} R^{n_k}(c)$. Finally, let $\omega(R):=\Omega(R)\cap \J$.

\begin{Def}[\cite{DU} p.200]
A rational map $R$ is called Misiurewicz or subexpanding if $R|_{\omega(R)}$ is expanding:
\[
\exists s\geq1, \, \exists \mu>1 \text{ such that } |(R^s)'(z)|\geq \mu, \forall z\in\omega(R).
\]
\end{Def}

\section{Dynamics on SG}

In \cite{dev} one investigates the rational maps of the form $z^n+\frac{\lambda}{z^m}$ with gasket-like Julia sets for $n\geq 2, m\geq 1$ and $\lambda \in \C$. A \textit{generalized Sierpinski gasket} is described as a compact subset of the closed unit disk, obtained by a similar process to SG by removing homeomorphic copies of $N$-polygons, having a $N$-fold symmetry, and from the second stage and onward of the construction, $m$ corners of a removed region lying in the boundary of one of the removed regions in the previous stage, with $1 \leq m < N$. For example, SG is homeomorphic to such a fractal with $N=3$ and $m=1$.
It is proven in \cite{dev} that the structure of a generalized SG for those described maps occurs, when they are so called $MS$-maps.

\begin{Def}[\cite{dev} Def. 2.2]\label{MS}
A map of the form $z^n+\frac{\lambda}{z^m}$ is called \textit{Misiurewicz-Sierpinski} map or shortly MS-map if 
\begin{itemize}
    \item each critical point lies in the boundary of the immediate basin of infinity, 
    \item each of the critical points is preperiodic.
\end{itemize}
\end{Def}

One should note that an $MS$-map is always Misiurewicz.
\begin{Prop}
\label{ms}
If all critical points of a rational map $R$ are preperiodic (without indifferent periodic points), then it is Misiurewicz.
\end{Prop}
\begin{proof}
Let $c_1, \cdots,c_n$ be the critical points of $R$, which reach a cycle of periods $p_1,\cdots,p_n$ and let $z_1,\cdots,z_n$ be any elements of the respective orbits. The union of the orbits will be the $\omega$-limit set. All points contained in $\omega(R)$ will satisfy $|(R^{p_i})'(z_i)|> 1$ because they are not indifferent and in the Julia set. Let $s=\prod_{i=1}^n p_i$. Similar to $(\ref{chain})$ one has
\[
(R^s)'(z_i)=\left( R'(R(z_i))\cdot R'(R^2(z_i))\cdots R'(R^n(z_i))\right)^{s/p_i}.
\]
All of these derivatives will have absolute value greater than one. Take $\mu$ to be the minimum of them.
\end{proof}

We will always assume $R$ to be an $MS$-map. Let $\beta_\lambda$ be the boundary, a simple closed curve as proven in \cite{dev}, of the immediate basin of infinity, i.e. the outer Fatou component where points tend to infinity.  And let $\tau_\lambda$ be the boundary of the neighbourhood of $0$ that is mapped to the basin of infinity, also called \textit{trap door}. The critical points are now exactly the intersection points of $\beta_\lambda$ and $\tau_\lambda$. Moreover, define $\tau^k_\lambda=R^{-k}(\tau_\lambda)$ which consist of several connected components and are the boundaries of the removed regions from the second step and onward of the construction of the Julia set. The Julia set is now the closure of $\beta_\lambda \cup \bigcup_{k\geq0}\tau^k_\lambda$. 

In Section 4 to 7, we will mainly focus on the case $R(z)=z^2+\frac{\lambda}{z}$ with $\lambda=-\frac{16}{27}$ where the resulting Julia set is homeomorphic to the standard Sierpinski gasket. The map has three critical points $c_0=-2/3$, $c_1=1/3+0.577i$ and $c_2=1/3-0.577i$. The point $z_0=R(c_0)=4/3$ is a fixed point with $R'(z_0)=3$, hence it lies on the Julia set. Moreover, $z_1=R(c_1)$ and $z_2=R(c_2)$ form a $2$-periodic cycle and $(R^2)'(z_1)=(R^2)'(z_2)=(R^2)'(z_0)=9$. Hence, by Proposition $\ref{ms}$ the map is Misiurewicz. 
The outer topological triangle of SG with vertices $z_0,z_1,z_2$ corresponds to $\beta_\lambda$, and the first removed triangle in the center of SG has vertices $c_0,c_1,c_2$. At the next step $\tau^1_\lambda$ consists of the three smaller topological triangles removed in the second step of the construction of SG. And exactly $m=1$ corners of these removed regions lie in the boundary of the removed region in the previous stage. 

This gives rise to a new construction of $SG$ with a dynamical background. One takes the same graph approximation as in the self-similar case defined by $\Gamma_0=\beta_\lambda$ and $\Gamma_{m+1}=\Gamma_m \cup \tau_\lambda^m$ together with their vertex set  $V_0=\{z_0,z_1,z_2\}$ and $V_{m+1}=R^{-1}V_{m}$, but with a completely different mapping. The mapping for $V_2$ as an example is shown in Figure \ref{v2}.

\begin{figure}
\begin{center}
\includegraphics[scale=0.08]{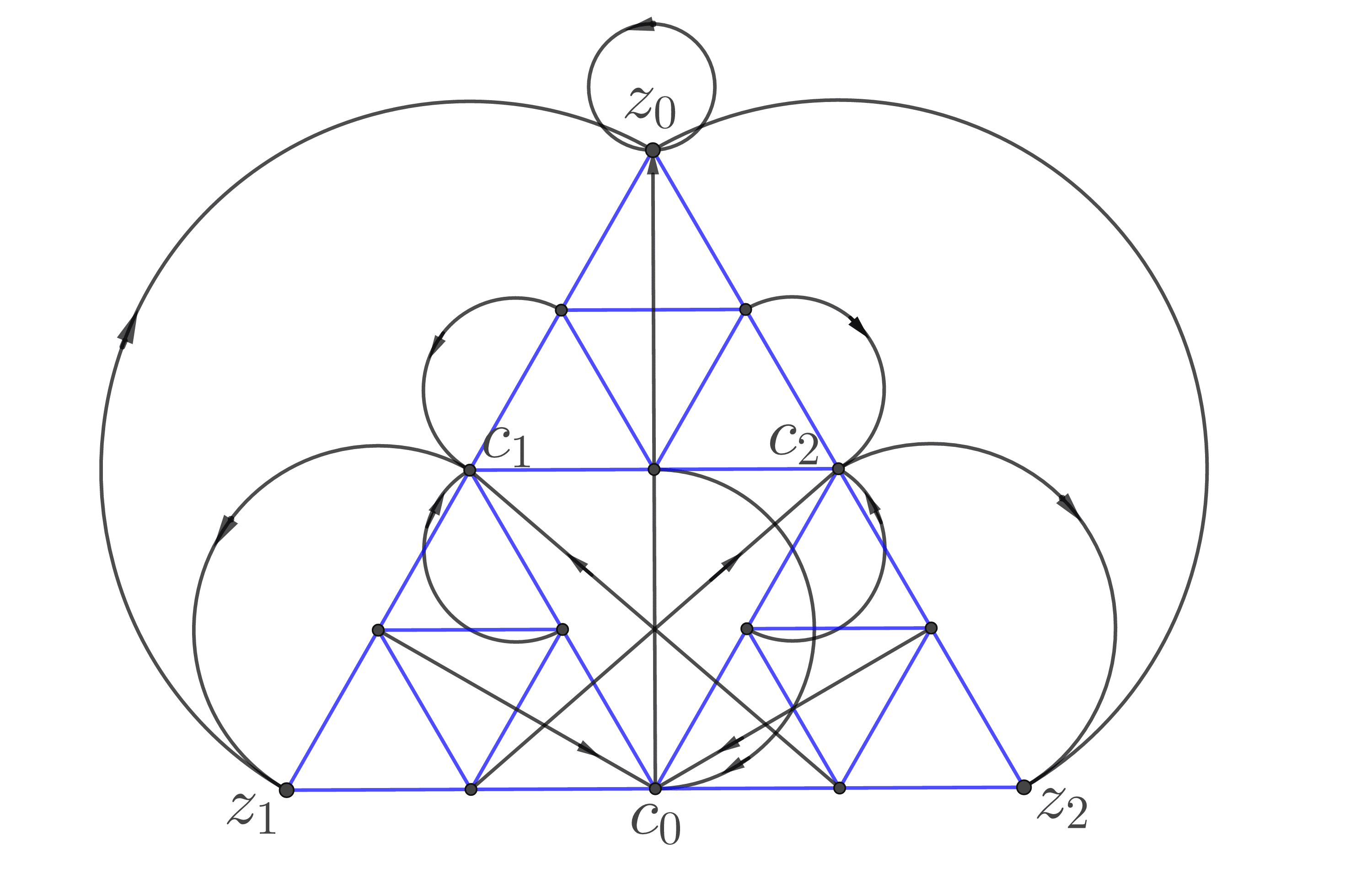}
\end{center}
\caption{Dynamics for $V_2$}
\label{v2}
\end{figure}

We should point out that, by a slight abuse of notation, we write SG for both the Julia set and the standard Sierpinski gasket since they are the same in the sense of homeomorphism. The notations $\Gamma_m, V_m$ are used in  the same way. In particular, we will also use $R$ to represent the dynamical map on the standard Sierpinski gasket inherited from the one on the Julia set.

\section{Standard Energy Form}

Since one has now a graph approximation, to construct an energy form on SG, the first step is to define the discrete graph energies:
\[
E_m(u,v)=\sum_{x\sim_m y}c_m(x,y)(u(x)-u(y))(v(x)-v(y)) 
\]
for functions $u,v:V_m \mapsto \R$ where $x\sim_m y$ means $x$ and $y$ are adjacent nodes in $V_m$ and $c_m(x,y)$ are called \textit{conductances}.

In order to respect the dynamics of $R$, we want to choose suitable conductances $c_m(x,y)$ such that the energy is invariant:
\begin{equation} \label{energyinvariance}
    E_m(u\circ R,v\circ R)=c\cdot E_{m-1}(u,v)
\end{equation}
with some constant $c$ independent of $u,v$ or $m$. We only need to consider the $u=v$ case by the polarization identity, see \cite{str}.

For $m=1$ one has:
\[
E_0(u)=c(z_0,z_1)(u(z_0)-u(z_1))^2+c(z_1,z_2)(u(z_1)-u(z_2))^2+c(z_2,z_0)(u(z_2)-u(z_0))^2,
\]
and
\begin{equation}
\begin{split}
    E_1(u\circ R) &= (c(z_1,c_1)+c(c_1,c_2)+c(c_2,z_2))(u(z_1)-u(z_2))^2 \\
    & + (c(c_1,z_0)+c(c_0,z_2)+c(c_0,c_1))(u(z_0)-u(z_1))^2 \\
    & + (c(z_0,c_2)+c(c_0,z_1)+c(c_0,c_2))(u(z_2)-u(z_0))^2,
\end{split}
\end{equation}
where we write $E_m(u)=E_m(u,u)$ for short.

There are multiple solutions for the conductances such that the invariance property \eqref{energyinvariance} is fulfilled (we will deal with this in Section $7$). The easiest solution would be to set all conductances to $1$. For higher levels, the property \eqref{energyinvariance} will still hold since the degree of the map is $3$, thus each point has $3$ preimages and always three of the $3^{m+1}$ edges in $V_m$ are identified. Hence, one obtains the identity
\begin{equation} \label{e3}
    E_m(u\circ R^k, v\circ R^k)=3^k E_{m-k}(u,v).
\end{equation}

On the other hand, since the graph energies $E_m$ are not different from the standard self-similar ones \cite{str}, to make them compatible, one still needs to renormalize $E_m$ to $\E_m$:
\begin{equation} \label{renorm}
    \E_m(u,v)=\left(\frac{3}{5}\right)^{-m}E_m(u,v), \forall u,v \in l(V_m).
\end{equation}
Here ``compatible'' means that we always have
\[
\E_{m-1}(u)=\min\{\E_m(v), \,v\in l(V_m), v|V_{m-1}=u \}, \forall u \in l(V_{m-1}),
\]
where $\E_m(u)=\E_m(u,u)$ for short. 
Call $v\in l(V_m)$ that attains the minimal energy the \textit{harmonic extension} of $u \in l(V_{m-1})$.

Combining \eqref{e3}, \eqref{renorm}, and passing $m$ to infinity one obtains an energy form $(\E, dom\E)$ on SG with
\[
dom \E =\{ u \in {C}(SG): \, \E(u)<\infty\}
\]
and
\[
\E(u,v)=\lim_{m\mapsto\infty} \E_m(u,v), \, \forall u,v \in dom\E,
\]
where the existence of the limit is due to the compatible requirement of $\E_m$.
Moreover, this form is invariant under $R$:

\begin{equation}
    \E(u\circ R^k,v\circ R^k)=5^k\E(u,v), \forall u,v \in dom\E.
\end{equation}

\section{Measure and Laplacian}

The next step is to define a suitable measure $\mu$ on SG. Again to respect the dynamics of $R$, it is reasonable to assume $\mu$ to be $R$-invariant, i.e.
\[
\mu(A)=\mu(R^{-1}A)
\]
for every Borel set $A$ and therefore
\begin{equation}\label{integral}
    \int_A f\circ R d\mu = \int_A f d\mu.
\end{equation}

A theorem by Denker and Urba\'{n}ski states that there is not much choice.

\begin{Prop}[\cite{DU}Theorem 4.6]
For a Misiurewicz rational map $R$ there exists a unique, ergodic, $R$-invariant probability measure $\mu$.
\end{Prop}

The following lemmas show that the standard measure $\mu$, assigning $1/3^m$ to each $m$-cell, on SG is the only suitable measure.

\begin{Lemma}
The standard measure $\mu$ is $R$-invariant.
\end{Lemma}
\begin{proof}
Every Borel set $A$ can be approximated arbitrarily well by a finite union of cells in SG, say $A\approx\bigcup C_i$ where the $C_i$ are $m$-cells. By the continuity of $R$ it suffices to show the invariance property for $C_i$. The preimage of every $m$-cell consists of three $(m+1)$-cells. Since a $(m+1)$-cell has one third of the measure of a $m$-cell, one may conclude
\[
\mu(\cup C_i)=\mu(\cup R^{-1}C_i).
\]
\end{proof}

\begin{Lemma}
The standard measure $\mu$ is ergodic.
\end{Lemma}
\begin{proof}
This follows from Theorem 4.6 and 4.7 in \cite{DU}.
\end{proof}

Thus, given the standard energy and measure on SG, by a standard argument \cite{str}, one obtains the standard Laplacian $\Delta$, defined by the weak formulation: for a function $u \in dom\E$, say $\Delta u=f$ with $f \in {C}(SG)$ if
\begin{equation}\label{weaklaplacian}
    -\E(u,v)=\int f\cdot v \,d\mu 
\end{equation}
holds for any $v \in dom\E$ with $v|V_0=0$.

Applying $(\ref{energyinvariance})$ one has $u\in dom\, \E \Rightarrow u\circ R\in dom\,\E$ and
\[ 
    \int \Delta (u\circ R)(v\circ R) \,d\mu = -\E(u\circ R,v\circ R)=-5\E(u,v)=5 \int (\Delta u)v \,d\mu.
\]
By the invariance of the measure $(\ref{integral})$,
\[
    5 \int (\Delta u)v \,d\mu=5\int ((\Delta u)v)\circ R \, d\mu = 5 \int ((\Delta u)\circ R) (v\circ R) \, d\mu.
\]
Hence,
\[
    \int \Delta (u\circ R)(v\circ R) \,d\mu=5 \int ((\Delta u)\circ R)( v\circ R) \, d\mu. \quad 
\]
Since this shall hold for all $v \in dom\,\E$ with $v|_{V_0}=0$, one can eliminate the integral and divide by $v \circ R$ to get
\begin{equation}\label{laplinv}
    \Delta(u\circ R)=5 (\Delta u)\circ R.
\end{equation}
This is an analogue of the classical Laplacian in $\R^2$, where $\Delta(f\circ\rho)=(\Delta f)\circ\rho$ if $\rho$ is an orthogonal transformation, meaning it preserves the inner product. And in our case the energy $\E(u,v)$ serves as the inner product which is $R$-invariant. 

One can now prove an important property about the spectrum of the Laplacian.
\begin{Prop}
If $u$ is an eigenfunction of $\Delta$ with eigenvalue $\lambda$, then $u\circ R$ is also an eigenfunction of $\Delta$ with eigenvalue $5 \lambda$. In particular, $5\Sigma \subset \Sigma$, where $\Sigma$ is the spectrum of $\Delta$.
\end{Prop}
\begin{proof}
Given an eigenfunction $u$ one obtains with \eqref{laplinv} that
\[
    -\Delta(u\circ R)=-5 (\Delta u)\circ R= -5\lambda(u\circ R).
\]
\end{proof}

\section{Iterated Function System}

Let $\{F_i\}_{i=0,1,2}$ be the iterated function system (IFS) of SG (by looking at SG as the standard Sierpinski gasket), i.e. $F_i z=\frac{1}{2}(z-q_i)+q_i, \, i=0,1,2$, where $q_0,q_1,q_2$ are the three vertices of a triangle. Then SG satisfies the self-similar identity
\[
SG=\bigcup_{i=0}^2 F_i SG.
\]
Same as we have done before, $\{F_i\}_{i=0,1,2}$ can also be interpreted as an IFS of the Julia set SG. In this sense, by an easy observation, one can find that 
\begin{equation}\label{rf}
R(z)=F_0^{-1}(z), \quad \forall z\in F_0SG.
\end{equation}

On the other hand, all maps of the form $z^2+\frac{\lambda}{z}$ have the symmetry properties:
\begin{equation}\label{f2}
    R(\omega z)=\omega^2 R(z),
\end{equation}
\begin{equation}\label{f3}
    R(\omega^2 z)=\omega R(z),
\end{equation}
where $\omega$ is the rotation of a third circle, $\omega=e^{\frac{2}{3}\pi i}$. Note that one can express the maps $F_1$ and $F_2$ in terms of $F_0$ and rotations by
\begin{equation}\label{f01}
    F_1^{-1}=\omega \circ F_0^{-1} \circ \omega^2,
\end{equation}
\begin{equation}\label{f02}
    F_2^{-1}=\omega^2 \circ F_0^{-1} \circ \omega.
\end{equation}

If $z \in F_0SG$, then $\Tilde{z}:=\omega z$ lies in $F_1SG$. Now $(\ref{f2})$ becomes
\[
R(\Tilde{z})=\omega^2 R(\omega^2 \Tilde{z}).
\]
Since $\omega^2 \Tilde{z}=z \in F_0 SG$, one may apply $(\ref{rf})$ to get
\[
R(\Tilde{z})=\omega^2 \circ F_0^{-1}\circ \omega^2(\Tilde{z}).
\]
And $(\ref{f01})$ yields
\begin{equation}\label{F1}
    R(\Tilde{z})=\omega \circ F_1^{-1}(\Tilde{z}).
\end{equation}
Similarly, for $\doubletilde{z}:=\omega^2 z \in F_2 SG$, equation $(\ref{f3})$ becomes

\[
R(\doubletilde{z})=\omega F_0^{-1}(\omega \doubletilde{z})=\omega^2\circ \omega^2 \circ F_0^{-1} \circ \omega (\doubletilde{z}).
\]
And $(\ref{f02})$ yields
\begin{equation}\label{F2}
    R(\doubletilde{z})=\omega^2 \circ F_2^{-1}(\doubletilde{z}).
\end{equation}

Thus by \eqref{rf}, \eqref{F1}, \eqref{F2}, the preimage $R^{-1}$ satisfies:

\begin{equation}\label{inverse0}
(R|_{F_0SG})^{-1}=F_0,
\end{equation}
\begin{equation}
(R|_{F_1SG})^{-1}=F_1 \circ \omega^2,
\end{equation}
\begin{equation}\label{inverse2}
(R^{-1}|_{F_2SG})^{-1}=F_2 \circ \omega.
\end{equation}

Hence, one obtains a new IFS $\{\Tilde{F}_i\}_{i=0}^2$ of SG with $\Tilde{F}_0=F_0, \Tilde{F}_1=F_1 \circ \omega^2$ and $\Tilde{F}_2=F_2\circ \omega $ since

\begin{equation}\label{IFS}
    SG=R^{-1}(SG)=\bigcup_{i=0}^2 \Tilde{F}_i SG.
\end{equation}

Note that for $V_m$ we introduced before, it holds that

\begin{equation}
    V_{m+1}=R^{-1}(V_m)=\bigcup_{i=0}^2 \Tilde{F}_i V_m.
\end{equation}

\section{Self-Similar Energy Forms}

In this section, we aim to describe the $R$-invariant energy forms on SG. Let $(\E,dom\E)$ be such a form. Then for any $u\in dom\E$, we have $u\circ R \in dom\E$ and $\E(u\circ R)=\rho\E(u)$ for some constant $\rho$ independent of $u$. Note that $u=u\circ R \circ \Tilde{F}_i$ for any $i=0,1,2$. This gives
\[
\E(u\circ R)=\rho \E(u)=\rho \sum_{i=0}^{2} a_i \E(u\circ R \circ \Tilde{F}_i)
\]
for any probability weight $(a_0,a_1,a_2)$. Since $u\circ R$ can run over all functions in $dom\E$, the following self-similar identity then holds:
\begin{equation}\label{rinvid}
    \E(u)=\sum_i r_i^{-1}\E(u\circ \Tilde{F}_i)
\end{equation}
with $r_i^{-1}=\rho a_i$.

Due to the above reason, in this section, we want to look at all the self-similar energy forms on SG satisfying \eqref{rinvid} with $r_i>0$, which in turn are $R$-invariant with $\rho=\sum_{i=0}^{2} r_i^{-1}$. Call $r_i$ the \textit{renormalization weight} as we did in the self-similar case \cite{str}.

Equivalently, one seeks a solution to the following renormalization problem. Given initial energy $\E_0$ with conductances $c(x,y)$ on $V_0$ one defines the energy $\E_1$ on $V_1$ by 
\begin{equation}\label{e7}
\E_1(u)=\sum_{x \sim_1 y} c_1(x,y)(u(x)-u(y))^2
\end{equation}
for
\[
c_1(\Tilde{F}_ix,\Tilde{F}_i y)=r_i^{-1} c(x,y)  \text{ if } \, x,y \in V_0.
\]
One says that $\E_0$ solves the renormalization problem with given weights  $r_i>0$ if there exists $\lambda>0$ such that
\begin{equation}\label{renormalization}
    \E_1(\Tilde{u})=\lambda^{-1} \E_0(u),\quad\forall u \in l(V_0),
\end{equation}
holds for the harmonic extension $\Tilde{u}$ on $V_1$. As a result, $\E_0$ and $\E_1$ are compatible. After finding out the constant $\lambda$, one corrects $r_i$ to $\Tilde{r}_i=\lambda^{-1}r_i$. Then the graph energy $\E_m$ can be defined in a similar way for higher levels which converge to an energy on SG. This problem has been well studied for SG with the standard IFS and the IFS composed with twists \cite{twi}.

In order to determine the existence of $\lambda$ in $(\ref{renormalization})$ one uses the electric network interpretation and  $\Delta - Y$ transformations \cite{str}.
Let $c_0,c_1,c_2$ be the initial conductances on $V_0$. Denote $w_i=c_i^{-1}$ the initial resistances on $V_0$. Since the weights will be corrected afterwards anyway, one may set $r_0=1$. 

In Figure $\ref{v0fig}$, the $\Delta - Y$ transformation is shown for $V_0$. Without loss of generality, one may set $\frac{w_1w_2}{D}=1$ and denote $s_1=\frac{w_0w_2}{D}$ and $s_2=\frac{w_0w_1}{D}$ where $D=w_0+w_1+w_2$. 

\begin{figure}
\begin{center}
\includegraphics[scale=0.1]{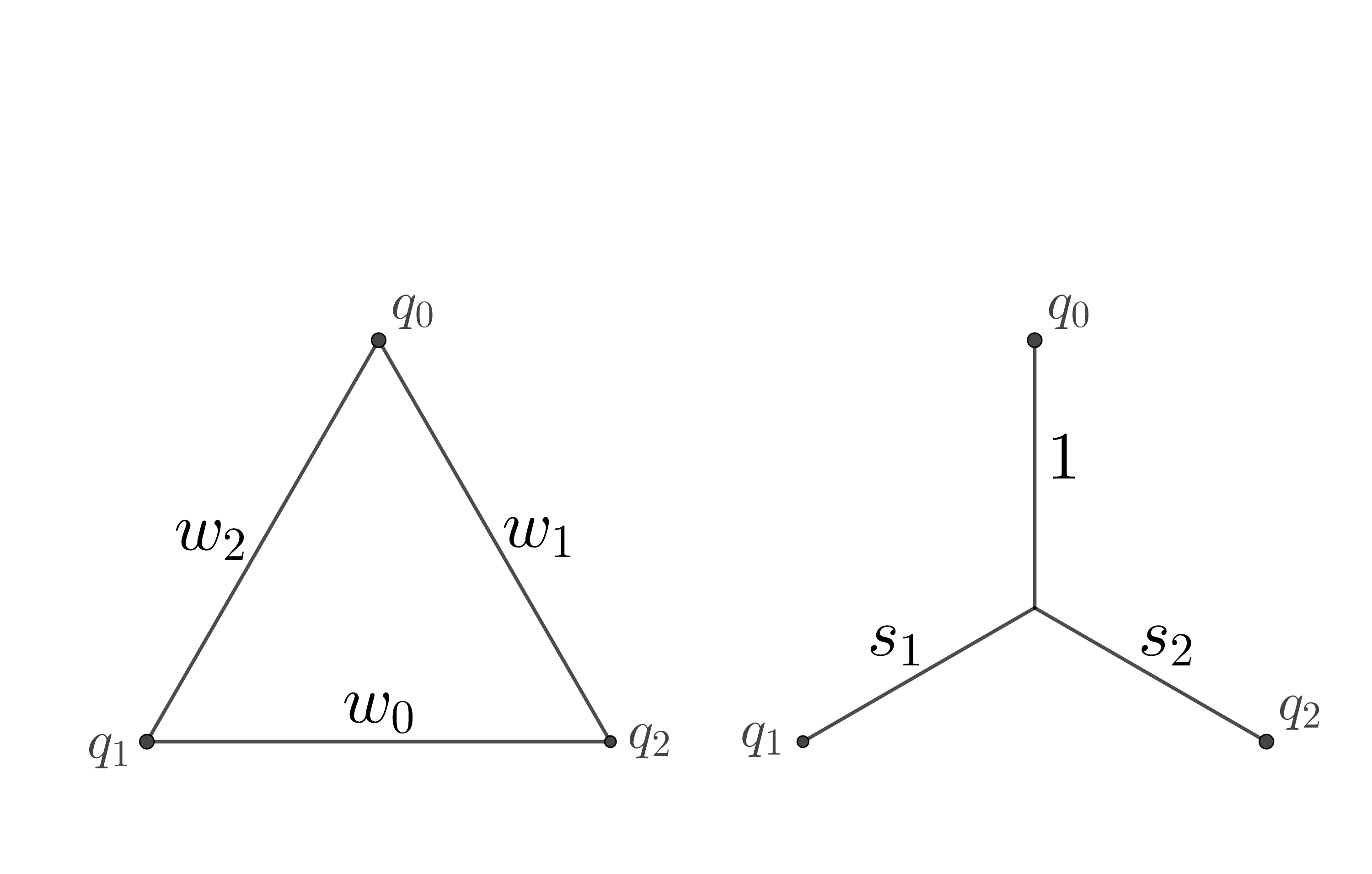}
\end{center}
\caption{Network transformation on $V_0$}
\label{v0fig}
\end{figure}

\begin{figure}[H]
\begin{center}
\includegraphics[scale=0.1]{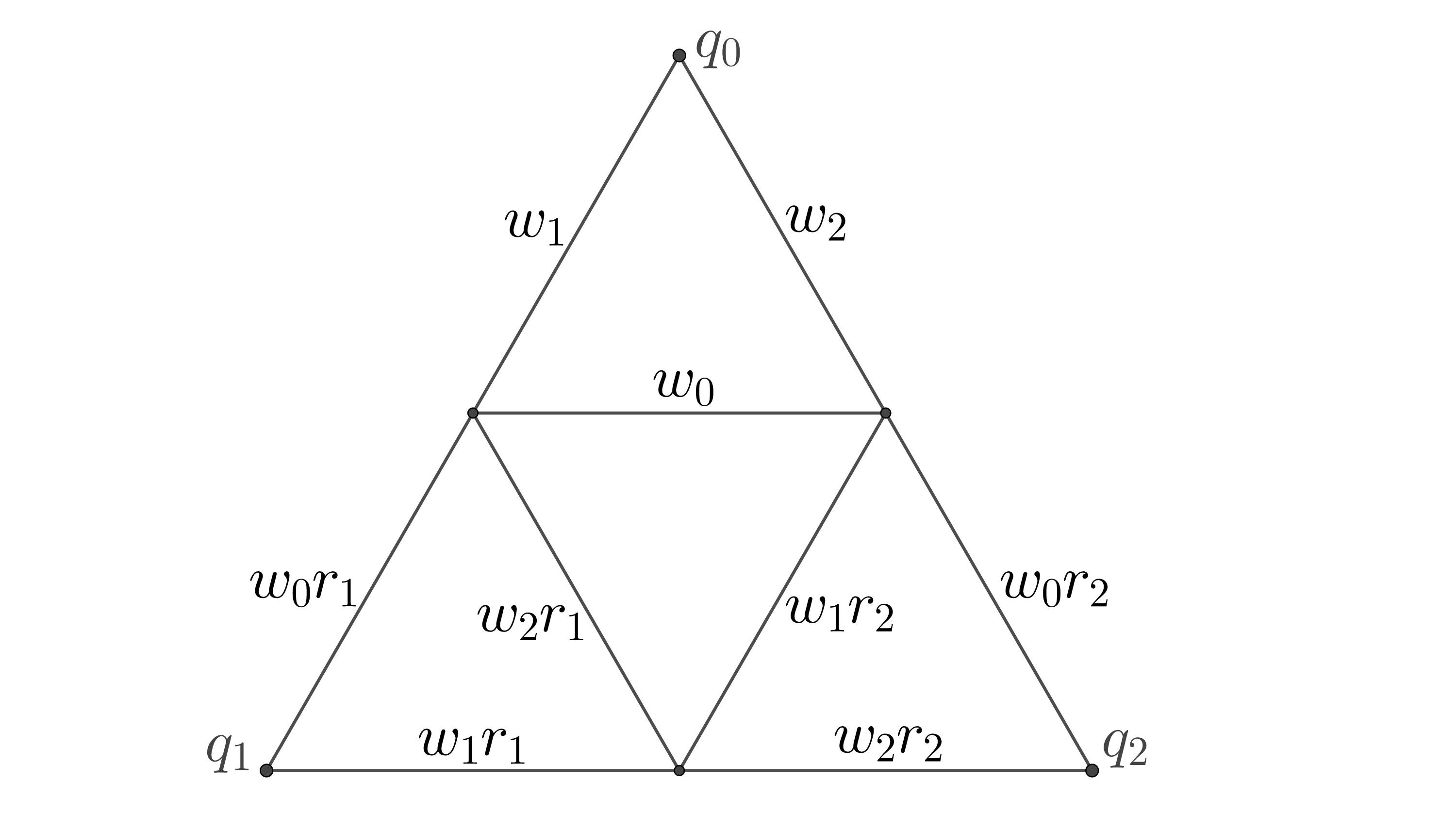}
\end{center}
\caption{Resistances on $V_1$}
\label{v1fig}
\end{figure}

\begin{figure}[H]
\begin{center}
\includegraphics[scale=0.13]{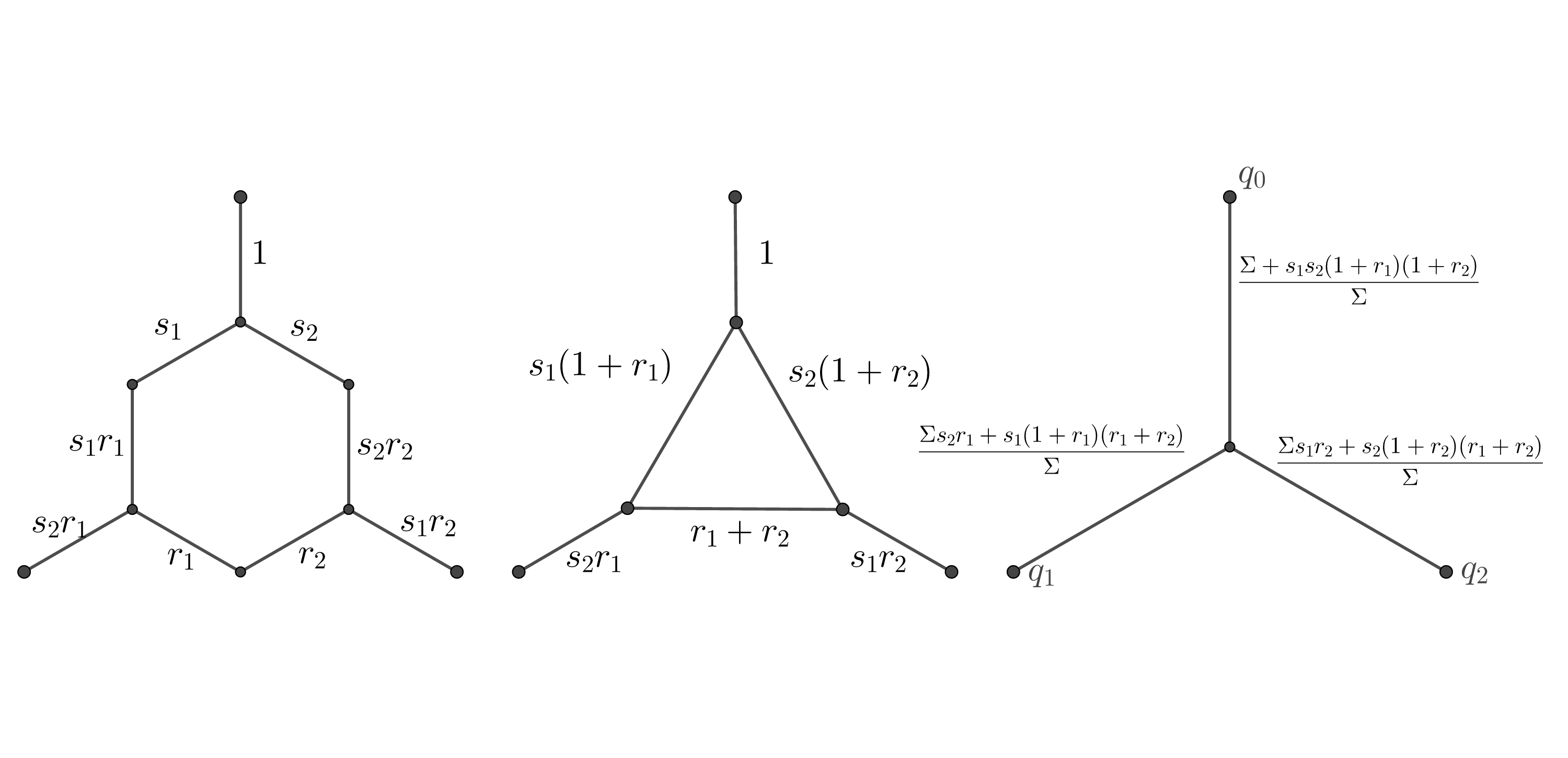}
\end{center}
\caption{Network transformation on $V_1$}
\label{step2}
\end{figure}

Corresponding to $(\ref{e7})$, the resistances on $V_1$ are shown in Figure $\ref{v1fig}$. Now one applies $\Delta - Y$ transformations as illustrated in Figure \ref{step2}. For simplicity, abbreviate $\Sigma=r_1+r_2+s_1+s_2+s_1 r_1+s_2 r_2$. The resulting network should be a multiple of the $Y$-network in Figure $\ref{v0fig}$ in accordance to $(\ref{renormalization})$. Hence, one obtains the system of equations:

\begin{equation}\label{mse1}
    \Sigma+s_1s_2(1+r_1)(1+r_2)=\lambda \Sigma,
\end{equation}
\begin{equation}\label{mse2}
    \Sigma s_2r_1+s_1(1+r_1)(r_1+r_2)=\lambda s_1\Sigma,
\end{equation}
\begin{equation}\label{mse3}
    \Sigma s_1r_2+s_2(1+r_2)(r_1+r_2)=\lambda s_2 \Sigma.
\end{equation}
One can use $(\ref{mse1})$ to determine $\lambda$. Since the weights and conductances are positive,  $\lambda>1$ and so $\Tilde{r}_0=\lambda^{-1}<1$. Moreover, one has $r_1<\lambda \frac{s_1}{s_2}$ and $r_2<\lambda \frac{s_2}{s_1}$, so at least one of $\Tilde{r}_1$ and $\Tilde{r}_2$ is smaller than $1$. The remaining equations are just:
\begin{equation}\label{mse5}
\Sigma s_2r_1+s_1(1+r_1)(r_1+r_2)=\Sigma s_1 + (1+r_1)(1+r_2)s_1^2s_2,
\end{equation}
\begin{equation}\label{mse6}
    \Sigma s_1r_2+s_2(1+r_2)(r_1+r_2)=\Sigma s_2 + (1+r_1)(1+r_2)s_1s_2^2.
\end{equation}

An easy observation is that $(\ref{mse5})$ is linear in $r_2$ and $(\ref{mse6})$ is linear in $r_1$. This also luckily occurs for SG with the IFS composed with twists \cite{twi}. Using equation $(\ref{mse5})$, one can plug in the expression for $r_2$ into $(\ref{mse6})$ to obtain one equation in three variables. Given initial conductances, one can investigate whether suitable weights exist. 

However, we could not solve this equation for one variable in the general case as no simple factorization like in \cite{twi} has been observed. Hence, we specify on a specific case that is respecting another symmetric identity of the map $R$:

\begin{equation}\label{sym_conj}
  R(\overline{z})=\overline{R(z)}.
\end{equation}
In terms of the energy we seek
\begin{equation}
 \E(u \circ \mathfrak{c})=\E(u), \quad \forall u \in dom\E,
\end{equation}
where $\mathfrak{c}$ denotes the reflection along the real axis (the line intersecting $z_0$ and $c_0$).

Looking at $V_0$ this means the following equations are equal:
\[
\E_0(u)=c_2(u(z_0)-u(z_1))^2+c_0(u(z_1)-u(z_2))^2+c_1(u(z_2)-u(z_0))^2
\]
\[
\E_0(u\circ\mathfrak{c})=c_2(u(z_0)-u(z_2))^2+c_0(u(z_2)-u(z_1))^2
+c_1(u(z_0)-u(z_1))^2.
\]
Thus $c_1=c_2$. This is also sufficient for higher levels and one can see that this implies $s_1=s_2:=s$, too. Then the right sides of $(\ref{mse5})$ and $(\ref{mse6})$ are equal and one obtains
\[
    \Sigma sr_1+s(1+r_1)(r_1+r_2)=\Sigma sr_2+s(1+r_2)(r_1+r_2),
\]
and thus
\[
\Sigma s(r_1-r_2)=s(r_1+r_2)(r_2-r_1).
\]
Since $\Sigma$ and all other variables are positive, one must have $r_1-r_2=0$. So one has equal weights $r:=r_1=r_2$. Equations $(\ref{mse5})$ and $(\ref{mse6})$ have now the same form:
\begin{equation}
    (2r+2sr+2s)sr+s(1+r)2r=(2r+2sr+2s)s+(1+r)^2s^3.
\end{equation}
After dividing by $s$, it simplifies to
\begin{equation}
    (r+1)^2s^2+(2-2r^2)s-4r^2=0
\end{equation}
with the two solutions
\[
    s_\pm=\frac{r-1\pm\sqrt{5r^2-2r+1}}{r+1}.
\]
For positive $r$, one has then the unique positive solution $s_+$. 

We summarize our results into a theorem.
\begin{Thm}
For any positive weights $(r_0,r_1,r_2)$ with $r_1=r_2$, there exists a unique positive $\lambda$, such that for the weights $\lambda^{-1}(r_0,r_1,r_2)=(\Tilde{r}_0,\Tilde{r}_1,\Tilde{r}_2)$ there is a unique (up to a constant multiple) nondegenerate energy form $(\E,dom\E)$ on SG 
satisfying
\[
\E(u)=\sum_{i=0}^2 \Tilde{r}_i^{-1} \E(u \circ \Tilde{F}_i), \quad \forall u \in dom\E.
\]
Moreover, $0 < \Tilde{r}_i < 1 \, for \, i=0,1,2$.
This produces all the symmetric and $R$-invariant energy forms on SG, i.e. 
\begin{equation}
    \E(u\circ \mathfrak{c})=\E(u) \text{ and } \E(u\circ R)=\sum_{i=0}^2 \Tilde{r}_i^{-1} \E(u), \quad \forall u \in dom\E.
\end{equation}
\end{Thm}

\section{The general case}

The discussion in \cite{dev} gives rise to a large class of Julia sets that have a generalized Sierpinski gasket structure. In this section we will generalize the results specific for $SG$ and find an IFS for a brought set of rational maps that will be a foundation to define energy forms and Laplacians on their Julia sets.

The first step is to look at the symmetries. 

\begin{Prop}
For a rational map $R(z)=z^n+\frac{\lambda}{z^m}$ where $n\geq2 $, $m\geq 1$ and $\lambda\in\C$, let $N:=n+m$ and $\omega_N=e^{\frac{2\pi}{N}i}$. Then for $i=1,...,N-1$ one has the symmetry identities:
\begin{equation}\label{generalsymmetry}
    R(\omega_N^iz)=\omega_N^{in}R(z).
\end{equation} 
\end{Prop}
\begin{proof}
\begin{align*}
    R(\omega_N^iz)&=\omega_N^{in}z^n+\omega_N^{-im}\frac{\lambda}{z^m}=\omega_N^{in}(z^n+\omega_N^{-im-in}\frac{\lambda}{z^m}) \\
                  &=\omega_N^{in}(z^n+\frac{\lambda}{z^m})=\omega_N^{in}R(z).
\end{align*}
\end{proof}

Furthermore, if $\lambda$ is real, then one has the symmetry with the complex conjugate as in $(\ref{sym_conj})$ and the same with the derivative $R'(z)$. As a consequence, the Julia set will have an $N$-fold symmetry, degree $N$ and $N$ critical points (except $0$ and $\infty$). If the map $R$ is an $MS$-map, then the critical points $c_0,...,c_{N-1}$ form the intersection of $\beta_\lambda$ and $\tau_\lambda$, where $\beta_\lambda$ is the boundary of the immediate basin of infinity and $\tau_\lambda$ is the Fatou component containing the origin. Then $\tau_\lambda^1=R^{-1}\tau_\lambda $ consists of $N$ components.  One continues naming the components in counter-clockwise order from $I_0$ to $I_{N-1}$.  As an example one can look at the Julia set in Figure $\ref{ms4}$, which has degree $4$.

\begin{figure}[H]
\begin{center}
\includegraphics[scale=0.8 ]{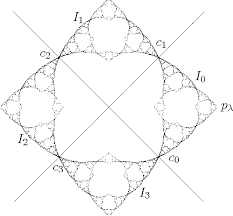}
\end{center}
\caption{The Julia set of $z^2+\frac{\lambda}{z^2}$ with $\lambda\approx-0.36428$ satisfying $R^4(c_0)=R^3(c_0)$}
\label{ms4}
\end{figure}

\begin{Thm}
The Julia set $\mathcal{J}$ of an $MS$-map $R(z)=z^n+\frac{\lambda}{z^m}, \lambda\in\mathbb{C}$, satisfies the self-similar identity:

\begin{equation}
\mathcal{J}=R^{-1}\mathcal{J}=\bigcup_{i=0}^{N-1}F_i\mathcal{J}
\end{equation}

with $F_i(z)=\omega_N^iF_0(\omega_N^{N-in }z)$ and $\omega_N=e^{\frac{2\pi}{N}i}$.
\end{Thm}
\begin{proof}
Let $R(z)=z^n+\frac{\lambda}{z^m}$ be an $MS$-map. By \cite{dl}, $\beta_\lambda$ and $\tau_\lambda$ are Jordan curves, and moreover
\begin{equation*}
\beta_\lambda\cap\tau_\lambda=\{c_i:~0\leq i\leq N-1\}.
\end{equation*}
Let $CV(R):=\{v_i=R(c_i):~0\leq i\leq N-1\}$ be the critical values of $R$. Then $CV(R)\subset\beta_\lambda\setminus \bigcup_{i=0}^{N-1}\{c_i\}$. We would like to mention that $v_i$ may be equal to $v_j$ if $i\neq j$. For example, if $n=m=2$, then $v_0=v_2$ and $v_1=v_3$. By Carath\'{e}odory's theorem (see \cite[p.\,20]{pom} or \cite[\S 17]{mil}), each critical value $v_i$ is the landing point of a unique external ray $\gamma_i$, where $0\leq i\leq N-1$. Then it is easy to see that
\begin{equation*}
D:=\mathbb{P}^1\setminus \bigcup_{i=0}^{N-1}\overline{\gamma}_i
\end{equation*}
consists of $N$ components $D_0$, $\cdots$, $D_{N-1}$, such that $I_i\subset D_i$. For $0\leq i\leq N-1$, let $F_i$ be the inverse branch of $R^{-1}$ such that $F_i(D)= D_i$.

Let $\widetilde{J}_i:=\overline{D}_i\setminus(B_\lambda\cup T_\lambda)$ where $B_\lambda$ is the immediate basin of infinity and $T_\lambda$ is the Fatou component containing the origin. Let $\widetilde{J}$ be the complement of $B_\lambda$. Then
\begin{equation*}
F_i\widetilde{J}=\widetilde{J}_i:~0\leq i\leq N-1.
\end{equation*}
 We denote by $\mathcal{J}_i:=F_i\mathcal{J}$. Then immediately (36) holds since $\mathcal{J}=\bigcup_{i=0}^{N-1}\mathcal{J}_i$.
 
 Given $z\in\mathcal{J}_0$, one has $\tilde{z}:=\omega^i_Nz\in\mathcal{J}_i$. By Proposition 5, 
 $$R(\tilde{z})=R(\omega^i_Nz)=\omega_N^{in}R(z)=\omega_N^{in}R(\omega_N^{N-i}\tilde{z}).$$
 Taking inverse of $R$, one then gets $F_i(z)=\omega^i_NF_0(\omega_N^{N-in}z)$. 
\end{proof}

The graph approximation of $\J$ will be defined as follows. Let $V_0$ be the post critical set, which by definition is a finite subset of $\beta_\lambda$, and denote $V_{m+1}=R^{-1}V_m$ inductively. It is easy to see that $V_{m} \subset V_{m+1}$. 
The graphs are  $\Gamma_0=\beta_\lambda$ and $\Gamma_{m+1}=\Gamma_{m}\cup  \tau_\lambda^{m}$ which are divided into several edges by the vertices $V_{m+1}$. 

\begin{figure}[H]
\begin{center}
\includegraphics[scale=0.1]{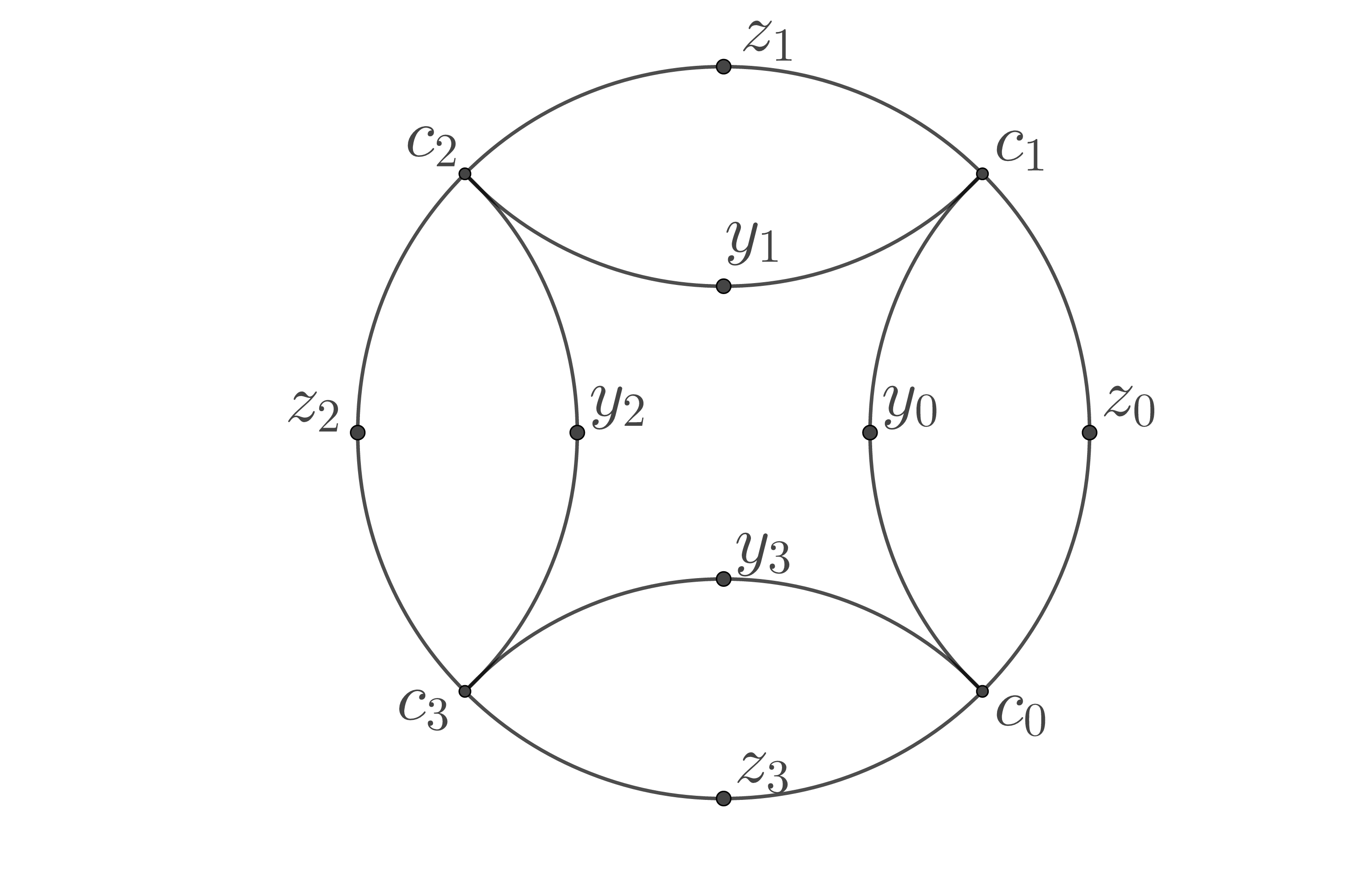}
\end{center}
\caption{The $\Gamma_1$ of the Julia set of $z^2+\frac{\lambda}{z^2}$ with $\lambda \approx-0.36428$}
\label{degree4_V_1}
\end{figure}

Note that $V_1\setminus V_0$ does not need to consist entirely of the critical points. For example, in the degree $4$ example as shown in Figure \ref{ms4}, there are additionally four other points lying in the middle of the four edges forming $\tau_\lambda$. The graph of $\Gamma_1$ can be seen in Figure \ref{degree4_V_1}. 

With the IFSs and the graph approximations for the generalized gasket-like Julia sets, it is then possible to construct the dynamically invariant energy forms and Laplacians in an equivalently self-similar way, as illustrated in the previous SG case. We leave the further discussion to interested readers.

\section*{Acknowledgments}
The authors wish to thank Shiping Cao and Fei Yang for helpful discussions and comments.

\end{document}